\documentclass[a4paper]{amsart}
\usepackage{amssymb, enumitem}
\usepackage[all]{xy}
\usepackage{hyperref, aliascnt}

\newcommand{\ca}{$C^*$-algebra}
\newcommand{\andSep}{\,\,\,\text{ and }\,\,\,}

\DeclareMathOperator{\charact}{char}

\newtheorem{lma}{Lemma}[section]

\newaliascnt{thmCt}{lma}
\newtheorem{thm}[thmCt]{Theorem}
\aliascntresetthe{thmCt}

\newaliascnt{corCt}{lma}
\newtheorem{cor}[corCt]{Corollary}
\aliascntresetthe{corCt}

\newaliascnt{prpCt}{lma}
\newtheorem{prp}[prpCt]{Proposition}
\aliascntresetthe{prpCt}

\theoremstyle{definition}

\newaliascnt{dfnCt}{lma}
\newtheorem{dfn}[dfnCt]{Definition}
\aliascntresetthe{dfnCt}

\newaliascnt{rmkCt}{lma}
\newtheorem{rmk}[rmkCt]{Remark}
\aliascntresetthe{rmkCt}

\newaliascnt{rmksCt}{lma}
\newtheorem{rmks}[rmksCt]{Remarks}
\aliascntresetthe{rmksCt}

\newaliascnt{exaCt}{lma}
\newtheorem{exa}[exaCt]{Example}
\aliascntresetthe{exaCt}

\newaliascnt{qstCt}{lma}
\newtheorem{qst}[qstCt]{Question}
\aliascntresetthe{qstCt}

\newaliascnt{pbmCt}{lma}
\newtheorem{pbm}[pbmCt]{Problem}
\aliascntresetthe{pbmCt}

\newaliascnt{ntnCt}{lma}
\newtheorem{ntn}[ntnCt]{Notation}
\aliascntresetthe{ntnCt}

\def\today{\number\day\space\ifcase\month\or   January\or February\or
   March\or April\or May\or June\or   July\or August\or September\or
   October\or November\or December\fi\   \number\year}

\newcounter{theoremintro}

\newaliascnt{thmIntroCt}{theoremintro}
\newtheorem{thmIntro}[thmIntroCt]{Theorem}
\aliascntresetthe{thmIntroCt}

\newaliascnt{dfnIntroCt}{theoremintro}
\newtheorem{dfnIntro}[dfnIntroCt]{Definition}
\aliascntresetthe{dfnIntroCt}

\title[Fully noncentral Lie ideals and invariant additive subgroups]{Fully noncentral Lie ideals and invariant additive subgroups in rings}
\date{\today}

\author[Eusebio Gardella]{Eusebio Gardella}
\address{Eusebio Gardella
Department of Mathematical Sciences, Chalmers University of
Technology and University of Gothenburg, Gothenburg SE-412 96, Sweden.}
\email{gardella@chalmers.se}
\urladdr{www.math.chalmers.se/~gardella}

\author{Tsiu-Kwen Lee}
\address{Tsiu-Kwen Lee,
Department of Mathematics,
National Taiwan University, Taipei 106, Taiwan.}
\email{tklee@math.ntu.edu.tw}

\author{Hannes Thiel}
\address{Hannes Thiel,
Department of Mathematical Sciences, Chalmers University of
Technology and University of Gothenburg, Gothenburg SE-412 96, Sweden.}
\email{hannes.thiel@chalmers.se}
\urladdr{www.hannesthiel.org}

\thanks{
The first named author was partially supported by the Swedish Research Council Grant 2021-04561.
The third named author was partially supported by the Knut and Alice Wallenberg Foundation (KAW 2021.0140).
}

\subjclass[2020]%
{Primary
16N60, 
16W10. 
Secondary
16S50, 
16W20, 
17B60. 
}

\keywords{Lie ideals, invariant subgroups, prime ideals, inner automorphisms, square-zero elements, commutators}

\date{\today}

\begin{document}

\begin{abstract}
We prove conditions ensuring that a Lie ideal or an invariant additive subgroup in a ring contains all additive commutators.
A crucial assumption is that the subgroup is fully noncentral, that is, its image in every quotient is noncentral. 

For a unital algebra over a field of characteristic $\neq 2$ where every additive commutator is a sum of square-zero elements, we show that a fully noncentral subspace is a Lie ideal if and only if it is invariant under all inner automorphisms.
This applies in particular to zero-product balanced algebras.
\end{abstract}

\maketitle

\section{Introduction}

Every associative ring $R$ carries a natural Lie ring structure with Lie product of two elements $x,y \in R$ defined as their commutator $[x,y] := xy - yx$.
A \emph{Lie ideal} in $R$ is then an additive subgroup $L \subseteq R$ such that $[R,L] \subseteq L$.
Obvious examples of Lie ideals in~$R$ are all additive subgroups that are either contained in the center $Z(R)$, or that contain the commutator subgroup $[R,R]$.
(Given subsets $G,H \subseteq R$ at least one of which is an additive subgroup, we follow the standard convention that $[G,H]$ and $GH$ denote the additive subgroups of $R$ generated by $\{[g,h]: g \in G, h \in H \}$ and $\{ gh : g \in G, h \in H \}$, respectively.)

In 1955, Herstein proved that this essentially already describes all Lie ideals in simple rings:
An additive subgroup $V$ of a simple ring $R$ is a Lie ideal if and only if either $V \subseteq Z(R)$ or $[R,R]\subseteq V$, unless $R$ has characteristic~2 and is 4-dimensional over its center;
see \cite[Theorem~5]{Her55LieJordanSimpleRing} and \cite[Theorem~1.5]{Her69TopicsRngThy}; 
see also Examples~\ref{exa:ExceptionalMatrixAlg} and~\ref{exa:ExceptionalDivAlg}.
This was subsequently developed into a comprehensive theory of Lie ideals in prime and semiprime rings \cite{Her70LieStructure, LanMon72LieStrPrimeChar2}. 

\medskip 

The deep results on Lie ideals in associative rings sparked the interest in results verifying that an additive subgroup in a ring is a Lie ideal.
A natural necessary condition is invariance under (certain) inner automorphisms.
Indeed, it is easy to see that every Lie ideal in a non-exceptional (\autoref{dfn:Exceptional}), simple ring is invariant under all inner automorphisms.
Further, if a subspace in an algebra over a field of characteristic~$\neq 2$ is a Lie ideal, then it is invariant under all inner automorphisms induced by square-zero elements;
see \autoref{prp:InvSqZero}.

The converse problem of when an additive subgroup that is invariant under (certain) inner automorphisms is automatically a Lie ideal was investigate by many authors, and we refer to the thorough introduction of Lanski's paper \cite{Lan89InvSubgpPrimeRg} for an overview.
We just mention the result of Amitsur \cite{Ami56InvSubmodSimpleRgs} that for a non-exceptional, simple algebra over a field of characteristic~$\neq 2$ that contains a nontrivial idempotent, a subspace is a Lie ideal if (and only if) it is invariant under inner automorphisms induced by square-zero elements.
This was extended to certain prime rings containing a nontrivial idempotent by Chuang \cite{Chu87InvAddSubgps}, and then to certain prime rings containing sufficiently many square-zero elements by Lanski \cite{Lan89InvSubgpPrimeRg}.
For non-exceptional, simple rings generated by their quasi-regular elements, and containing sufficiently many square-zero elements, Chuang \cite{Chu99InvSubgpSimpleRg} showed that an additive subgroup is a Lie ideal if (and only if) it is invariant under inner automorphisms.

Analogous questions for (closed) linear subspaces in operator algebras have also been extensively studied \cite{BreKisShu08LieIdeals, Mar10ProjCommutLieIdls, Rob16LieIdeals}.
In particular, it was shown by Marcoux and Murphy that a closed subspace in a \ca{} is a Lie ideal if and only if it is invariant under conjugation by unitaries \cite{MarMur98UniInvSpaceCAlg}.

\medskip

In this paper, we aim at conditions ensuring that an additive subgroup $V$ satisfies $[R,R] \subseteq V$ (and therefore is a Lie ideal).
We note that invariance under inner automorphisms of $R$ is a necessary condition:
Given an invertible element $u$ in the minimal unitization of $R$, and $x \in V$, we have
\[
uxu^{-1} = [u,xu^{-1}] + x \in [R,R] + V \subseteq V.
\]

We focus on rings $R$ for which $[R,R]$ is full.
Here we say that a subset~$X$ in~$R$ is \emph{full} if it generates $R$ as an ideal, that is, if it is not contained in a proper ideal of~$R$.
Rings that are generated by their commutators were studied in \cite{Bax65CommutatorSubgroupRing, Mes06CommutatorRings, Ero22SubrgGenCommutators, GarThi25GenByCommutators}.

In \cite{ChaRob23AutoContGrUnit}, Chand and Robert define a subset $X$ of a \ca{} $A$ to be `fully noncentral'  if $[A,X]$ is full in the sense that it generates $A$ as a closed ideal.
We adopt this terminology to the algebraic setting:

\begin{dfnIntro}
\label{dfn:FullyNoncentral}
We say that a subset $X$ in a ring $R$ is \emph{fully noncentral} if $[R,X]$ is full, that is, $[R,X]$ is not contained in a proper ideal of $R$. 
\end{dfnIntro}

Our main result shows that for a large class of algebras, a fully noncentral subspace is a Lie ideal if and only if it is invariant under inner automorphisms.
The following is \autoref{prp:TFAE-FNSubspSqZero}.

\begin{thmIntro}
\label{thmCharInv}
Let $A$ be an algebra over a field $\neq \{0,1\}$ such that every commutator in $A$ is a sum of square-zero elements, and every proper ideal is contained in a non-exceptional prime ideal.
Let $V \subseteq A$ be a fully noncentral subspace.
Then the following are equivalent:
\begin{enumerate}
\item
The subspace $V$ is invariant under all inner automorphisms of~$A$.
\item
The subspace $V$ is invariant under all inner automorphisms induced by square-zero elements of~$A$.
\item
We have $[A,A] \subseteq V$.
\item
The subspace $V$ is a Lie ideal.
\item
The subspace $V$ is an $[A,A]$-submodule.
\end{enumerate}
\end{thmIntro}

We note that for unital algebras over a field of characteristic $\neq 2$, every proper ideal is contained in a non-exceptional prime ideal.
Further, if an algebra is \emph{zero-product balanced} in the sense of \cite{GarThi23ZeroProdBalanced} (a notion closely related to the concept of a \emph{zero-product determined} algebra \cite{Bre21BookZeroProdDetermined}), then every commutator is a sum of square-zero elements (\cite[Theorem~5.3]{GarThi23ZeroProdBalanced}).
In particular, for a zero-product balanced algebra over a field of characteristic $\neq 2$, a fully noncentral subspace $V$ is a Lie ideal if and only $V$ is invariant under all inner automorphisms of $A$, if and only if $[A,A] \subseteq V$;
see \autoref{prp:Balanced}.
The class of zero-product balanced algebras contains all algebras generated by idempotents (\cite[Example~3.7]{GarThi23ZeroProdBalanced}), in particular, matrix algebras as well as simple algebras containing a nontrivial idempotent.

The proof of \autoref{thmCharInv} relies on the general result that if a subspace $V$ of an algebra $A$ over a field $\neq \{0,1\}$ is invariant invariant under all inner automorphisms, then $[x,V] \subseteq V$ for every square-zero element $x$;
see \autoref{prp:InvSqZero}.
If one additionally assumes that the field is infinite, then one can deduce that $[x,V] \subseteq V$ for every nilpotent element $x$;
see \autoref{prp:InvNilpot}.
As a consequence, we obtain a variation of \autoref{thmCharInv}, where the stronger assumption of working over an infinite field allows us to relax the condition of writing every commutator as a sum of square-zero elements to a sum of nilpotent elements;
see \autoref{prp:TFAE-FNSubspNilpot}.
Using that the subspace generated by nilpotent elements is invariant under inner automorphisms, we obtain as an application that additive commutators of nilpotent elements in an algebra over an infinite field are sums of nilpotent elements;
see \autoref{prp:CommutatorNilpotents}.

\subsection*{Acknowledgements}

The first and last named authors thank Laurent Marcoux and Leonel Robert for valuable comments on the Lie theory of \ca{s}.

\section{Results about general rings}

In this section, we devise a method that, for a general additive subgroup~$V$ in a ring~$R$, constructs an ideal $I \subseteq R$ satisfying $[R,I] \subseteq V$;
see \autoref{prp:LieIdealInSubgroup}.
The ideal $I$ is built from higher-order (generalized) normalizers of $V$ (see \autoref{ntn:T}), and in general it may very well happen that $I=\{0\}$.
If $V$ is a Lie ideal of $[R,R]$, or more generally an $[R,R]$-submodule, then we have better control over $I$;
see \autoref{prp:CaseRRSubmodule}.

\medskip

Given an additive subgroup $V$ in a ring $R$, the normalizer (with respect to the Lie product) is $\{ x \in R : [V,x] \subseteq V\}$.
We consider the closely related set $\{ x \in R : [R,x] \subseteq V\}$, which seems to have been considered first by Herstein in \cite{Her55LieJordanSimpleRing}, see also \cite[p.5]{Her69TopicsRngThy}, with the notation~$T(V)$.
This set also played a crucial role of the analysis of Lie ideals in \cite{BreKisShu08LieIdeals}, where it is denoted by $N(V)$.
We follow Herstein's notation, and also introduce higher-order versions:

\begin{ntn}
\label{ntn:T}
Let $R$ be a ring, and let $V \subseteq R$ be an additive subgroup.
We set
\[
T(V) := \big\{ x \in R : [R,x] \subseteq V \big\}.
\]

We inductively define $T^n(V)$ for $n \geq 1$ by setting $T^1(V) := T(V)$ and 
\[
T^{n+1}(V) := T(T^n(V)).
\]
\end{ntn}

Note that an additive subgroup $V \subseteq R$ is a Lie ideal if and only if $V \subseteq T(V)$.

The next result is folklore and follows for example from the proof of \cite[Lemma~3]{Her55LieJordanSimpleRing} or \cite[Lemma~1.4]{Her69TopicsRngThy}.
We include the short argument for the convenience of the reader.

\begin{lma}
\label{prp:T-Subring}
Let $R$ be a ring, and let $V \subseteq R$ be an additive subgroup.
Then $T(V)$ is a subring. 
\end{lma}
\begin{proof}
Using the bi-additivitiy of the Lie product, we see that $T(V)$ is an additive subgroup.
To show that $T(V)$ is closed under multiplication, let $x,y \in T(V)$.
Using that $[a,xy] = [ax,y] + [ya,x]$ for all $a \in R$, we have
\[
[R,xy] \subseteq [R,y] + [R,x] \subseteq V+V = V,
\]
and thus $xy \in T(V)$.
\end{proof}

Given a ring $R$, we use $\widetilde{R}$ to denote its minimal unitization, given by $\widetilde{R}=R$ if $R$ is unital, and by $\widetilde{R} = \mathbb{Z} \times R$ with coordinatewise addition and multiplication $(m,x)(n,y)=(mn,my+nx+xy)$ if $R$ is non-unital.
The map $R \to \widetilde{R}$ given by $x \mapsto (x,0)$, identifies $R$ with an ideal in $\widetilde{R}$.

The ideal of $R$ generated by an additive subgroup $V \subseteq R$ is $V+RV+VR+RVR$, which agrees with $\widetilde{R}V\widetilde{R}$.
Note that $RVR$ is also an ideal of $R$, but if $R$ is non-unital then it may not contain $V$ and therefore can be strictly smaller than $\widetilde{R}V\widetilde{R}$.

\begin{thm}
\label{prp:LieIdealInSubgroup}
Let $R$ be a ring, and let $V \subseteq R$ be an additive subgroup.
Then 
\[
\widetilde{R}\big[ T(V) \cap V, T^2(V) \cap T(V) \big]\widetilde{R} 
\ \subseteq \ V + V^2
\]
and
\[
\big[ R, \widetilde{R}[ T^2(V) \cap T(V), T^3(V) \cap T^2(V) ]\widetilde{R} \big] 
\ \subseteq \ V.
\]

In particular, if $[T^2(V) \cap T(V), T^3(V) \cap T^2(V)]$ is full, then $[R,R] \subseteq V$, and so~$V$ is a Lie ideal in $R$.
\end{thm}
\begin{proof}
Let $a,b \in \widetilde{R}$, let $x \in T(V) \cap V$, and let $y \in T^2(V) \cap T(V)$.
Note that $[\widetilde{R},X] = [R,X]$ for every subset $X \subseteq R$.
Using a direct computation in the first step, we get
\begin{align*}
a[x,y]b 
&= \big[ ax, [y,b] \big]
+ \big[ [y,b], a \big]x
+ \big[ a[x,b], y \big] \\
&\quad + [y,a][x,b]
+ [abx,y]
+ [y,ab]x \\
&\in \big[ R, [T^2(V),\widetilde{R}] \big]
+ \big[ [T^2(V),\widetilde{R}],\widetilde{R} \big]V
+ \big[ R, T(V) \big] \\
&\quad +[T(V),\widetilde{R}][T(V),\widetilde{R}]
+ [R,T(V)]
+ [T(V),\widetilde{R}]V \\
&\subseteq V + V^2.
\end{align*}
This verifies the first claimed inclusion.

Applying this inclusion for $T(V)$ in place of $V$, and using at the last step that~$T(V)$ is a subring by \autoref{prp:T-Subring}, we get
\[
\widetilde{R}\big[ T^2(V) \cap T(V), T^3(V) \cap T^2(V) \big]\widetilde{R}
\ \subseteq \ T(V) + T(V)^2
\ \subseteq \ T(V).
\]

It follows that
\[
\big[ R, \widetilde{R}[ T^3(V) \cap T^2(V), T^2(V) \cap T(V) ]\widetilde{R} \big]
\ \subseteq \ [ R, T(V) ]
\ \subseteq \ V,
\]
as desired.
\end{proof}

If we apply \autoref{prp:LieIdealInSubgroup} to a Lie ideal, then we obtain the following well-known result;
see, for example, \cite[Lemma~1.1]{Rob16LieIdeals}, \cite[Lemma~2.1(ii)]{Lee22AddSubgpGenNCPoly}.

\begin{cor}
\label{prp:CaseLieIdeal}
Let $R$ be a ring, and let $L \subseteq R$ be a Lie ideal.
Then $T^{n}(L)\subseteq T^{n+1}(L)$ for all $n\geq 0$, 
and we deduce that
\[
\widetilde{R}[L,L]\widetilde{R} \subseteq L + L^2, \andSep
\big[ R, \widetilde{R}[L,L]\widetilde{R} \big] \subseteq L.
\]
\end{cor}
\begin{proof}
Given additive subgroups $V_1 \subseteq V_2 \subseteq R$, we have $T(V_1) \subseteq T(V_2)$.
Since $L$ is a Lie ideal, we have $L \subseteq T(L)$ by definition.
Applying the above observation inductively, 
we get the desired inclusion.
The other inclusions now follow from \autoref{prp:LieIdealInSubgroup}.
\end{proof}

One says that an additive subgroup $V$ of a ring $R$ is an \emph{$[R,R]$-submodule} if $[[R,R],V] \subseteq V$.
This includes Lie ideals in $[R,R]$, but it also allows for subgroups~$V$ that are not contained in $[R,R]$.
Given an additive subgroup $V \subseteq R$, we set $V^{(0)} := V$, $V^{(1)} := [V,V]$, and $V^{(n+1)} := [V^{(n)},V^{(n)}]$ for $n \geq 1$.

In the next result, we show that for an $[R,R]$-submodule $V$, the groups $V^{(n)}$, for $n\geq 1$, form a decreasing sequence.
We do not claim that $V^{(1)} \subseteq V=V^{(0)}$.

\begin{lma}
\label{prp:T-RRSubmodule}
Let $V \subseteq R$ be an $[R,R]$-submodule and let $n\geq 1$.
Then $V^{(n)}$ is an $[R,R]$-submodule satisfying $V^{(n)} \subseteq T^n(V)$.
Further, we have $V^{(1)} \supseteq V^{(2)} \supseteq \cdots$.
\end{lma}
\begin{proof}
Claim~1: \emph{Given an $[R,R]$-module $W \subseteq R$, we have $[W,W] \subseteq T(W)$.}
Indeed, we have $[[R,R],W] \subseteq W$ by assumption.
Applying the Jacobi identity at the first step, we get
\[
\big[ R, [W,W] \big] 
\subseteq \big[ W, [R,W] \big] + \big[W, [W,R] \big]
\subseteq \big[ W, [R,R] \big] 
\subseteq W.
\]
and thus $[W,W] \subseteq T(W)$, which proves the claim.

Claim~2: \emph{Given an $[R,R]$-module $W\subseteq R$, then $[W,W]$ is an $[R,R]$-submodule as well.}
Applying the Jacobi identity at the first step, and using Claim~1 at the second step, we get
\[
\big[ [R,R], [W,W] \big] 
\subseteq \big[ [R,[W,W]], W \big] + \big[ [[W,W],R], W \big],
\subseteq [W,W].
\]
which verifies the claim.

Now, applying Claim~2 successively, we obtain that $V^{(n)}$ is an $[R,R]$-submodule for all $n \geq 1$.
Further, we deduce that
\[
V^{(n+2)}
= \big[ V^{(n+1)}, V^{(n+1)} \big] 
= \big[ [V^{(n)},V^{(n)}], V^{(n+1)} \big] 
\subseteq \big[ [R,R], V^{(n+1)} \big] 
\subseteq V^{(n+1)}
\]
for all $n \geq 0$.
Thus, we have $V^{(1)} \supseteq V^{(2)} \supseteq \ldots$.

Next, we verify by induction that $V^{(n)} \subseteq T^n(V)$ for all $n \geq 1$.
We have $V^{(1)} = [V,V] \subseteq T(V) = T^1(V)$ by Claim~1, which verifies the case $n=1$.
Assume that $V^{(n)} \subseteq T^n(V)$ for some $n \geq 1$.
Applying Claim~1 at the second step, we get
\[
\big[ R, V^{(n+1)} \big]
= \big[ R, [V^{(n)},V^{(n)}] \big] 
\subseteq V^{(n)}
\subseteq T^n(V)
\]
and thus $V^{(n+1)} \subseteq T(T^n(V)) = T^{(n+1)}(V)$.
\end{proof}

We obtain another corollary of \autoref{prp:LieIdealInSubgroup}.

\begin{cor}
\label{prp:CaseRRSubmodule}
Let $R$ be a ring, and let $V \subseteq R$ be an $[R,R]$-submodule.
Then 
\[
V^{(n)}\subseteq \bigcap_{j=1}^n T^j(V)
\]
for all $n\geq 1$, and we deduce that
\[
\widetilde{R}\big[ V^{(1)} \cap V, V^{(2)} \big]\widetilde{R} \subseteq V + V^2, \andSep
\big[ R, \widetilde{R}[ V^{(2)}, V^{(3)} ]\widetilde{R} \big] \subseteq V.
\]

In particular, if $[ V^{(2)}, V^{(3)} ]$ is full, then $[R,R] \subseteq V$, and so~$V$ is a Lie ideal in~$R$.
\end{cor}
\begin{proof}
The first inclusion follow from \autoref{prp:T-RRSubmodule}.
The other inclusions then follow from \autoref{prp:LieIdealInSubgroup}.
\end{proof}

\begin{rmk}
\label{rmk:V2V3-full}
Let $V \subseteq R$ be an $[R,R]$-submodule.
In light of \autoref{prp:CaseRRSubmodule}, it is interesting to determine when $[ V^{(2)}, V^{(3)} ]$ is full.
By \autoref{prp:T-RRSubmodule}. we have
\[
[ V^{(2)}, V^{(3)} ]
\ \subseteq \ [ V^{(2)}, V^{(2)} ]
\ = \ V^{(3)}
\ \subseteq \ V^{(1)}
\ = \ [V,V]
\ \subseteq \ [R,V].
\]

Thus, a necessary condition is that $[R,V]$ is full, that is, $V$ is fully noncentral (\autoref{dfn:FullyNoncentral}).
In the next section, we study rings for which full noncentrality of~$V$ is also sufficient.
\end{rmk}

\section{Rings where proper ideals are contained in prime ideals}
\label{sec:CofinalPrimeIdeals}

In this section, we study rings where every proper ideal is contained in a prime ideal.
For a Lie ideal $L$ in such a ring $R$ for which $[L,L]$ is full, we show that $L$ contains the commutator subgroup $[R,R]$ and that $R=L^2$;
see \autoref{prp:LLFull}.

\medskip

In \cite{Lee22AddSubgpGenNCPoly,Lee22HigherCommutators}, the second named author initiated the study of rings where every proper ideal is contained in a maximal ideal.
This class includes all rings that are unital or just finitely generated as an ideal, as well as all rings satisfying the ascending chain condition for ideals.
In \autoref{prp:MaxVsPrime} below we clarify the relationship with the class of rings where every ideal is contained in a prime ideal.
Note that there exist rings with maximal ideals that are not prime.
Further, there exist rings where every proper ideal is contained in a prime ideal, but not every proper ideal is contained in a maximal ideal, for example the commutative \ca{} $C_0(\mathbb{R})$ of continuous functions $\mathbb{R}\to\mathbb{C}$ which vanish at infinity. 
We refer the reader to \cite{BeiMarMik96RgsGenIds} and \cite[Chapter~4]{Lam01FirstCourse2ed} for the basic theory of prime ideals and prime rings.
We also note that the class of rings where every proper ideal is contained in a prime ideal (even a non-exceptional prime ideal) includes all unital rings, as well as every \ca{} \cite{GarThi24PrimeIdealsCAlg}.

A ring $R$ is said to be \emph{idempotent} if $R = R^2$, that is, if every element of~$R$ is the sum of products of elements of $R$.
This holds for 
every unital ring, as well as for Banach algebras with bounded approximate identity by the Cohen factorization theorem.

\begin{prp}
\label{prp:MaxVsPrime}
Let $R$ be a ring.
Then the following hold:
\begin{enumerate}
\item
If every proper ideal is contained in a prime ideal, then $R$ is idempotent. 
\item
If $R$ is idempotent, then every maximal ideal in $R$ is prime.
\item
If every maximal ideal in $R$ is prime, and every proper ideal is contained in a maximal ideal, then every proper ideal in $R$ is contained in a prime ideal.
\end{enumerate}
In particular, if every proper ideal of $R$ is contained in a maximal ideal, then every maximal ideal of $R$ is a prime ideal if and only if~$R$ is idempotent.

The implications are shown in the following diagram:
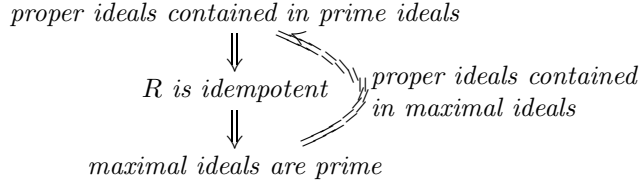
\[
\xymatrix@R-10pt{
\text{proper ideals contained in prime ideals} \ar@{=>}[d] \\
\text{$R$ is idempotent} \ar@{=>}[d] \\
\text{maximal ideals are prime} \ar@/_4pc/@{==>}[uu]_{\parbox{3.5cm}{proper ideals contained in maximal ideals}}
}
\]
\end{prp}
\begin{proof}
(1)
Suppose on the contrary that $R \neq R^2$. 
Then $R^2$ is a proper ideal, and thus $R^2\subseteq P$ for some prime ideal $P$ of $R$.
This implies that $R=P$, a contradiction.

(2) Let $M$ be a maximal ideal of $R$. Since $R=R^2$, we get $(R/M)^2\ne 0$.
It follows that $R/M$ is a simple ring and so it is a prime ring. Thus $M$ is a prime ideal of $R$.

(3) This is clear.
\end{proof}

For later use, we recall basic results about Lie ideals in rings.
These are well-known, and we include the short proofs for the convenience of the reader.

\begin{lma}
\label{prp:BasicLie}
Let $L$ be a Lie ideal in a ring $R$.
Then
\begin{enumerate}
\item
We have $\widetilde{R}L\widetilde{R} = \widetilde{R}L$.
\item
We have $[R,L^2] \subseteq [R,L]$.
\item
We have $\widetilde{R}[L,L]\widetilde{R} \subseteq L + L^2$.
\item
We have $\widetilde{R}[L,L^2]\widetilde{R} \subseteq L^2$.
\end{enumerate}
\end{lma}
\begin{proof}
We will use that $[\widetilde{R},V]=[R,V]$ for every additive subgroup $V \subseteq \widetilde{R}$.
In the proof of~(3) and~(4), we will use that $U[V,W] \subseteq [UV,W] + [W,U]V$ for additive subgroups $U,V,W \subseteq R$.

(1) The inclusion $\widetilde{R}L \subseteq \widetilde{R}L\widetilde{R}$ is clear.
We therefore have
\[
\widetilde{R}L\widetilde{R}
\subseteq [\widetilde{R}L,\widetilde{R}] + \widetilde{R}^2L
\subseteq \widetilde{R}[L,\widetilde{R}] + [\widetilde{R},\widetilde{R}]L + \widetilde{R}^2L
\subseteq \widetilde{R}L.
\]

(2)
Let $x,y \in L$ and $a \in R$.
Then
\[
[a,xy]
= axy - xya
= axy - yax + yax - xya
= [ax,y]+[ya,x]
\in [R,L].
\]

In fact, if $V$ is any additive subgroup of $R$, then the same argument shows that $[a,x] \in [R,V]$ for every $a \in R$ and every $x$ in the subring of $R$ generated by $V$.

(3)
This was already proved in \autoref{prp:CaseLieIdeal}.
Let us give an alternative proof here.
Since $[L,L]$ is a again a Lie ideal, we can apply~(1) at the first step, and get
\[
\widetilde{R}[L,L]\widetilde{R} 
= \widetilde{R}[L,L]
\subseteq [\widetilde{R}L,L] + [L,\widetilde{R}]L
\subseteq L + L^2,
\]

(4)
Proceeding at the first two steps as in~(3), and using that $[R,L^2] \subseteq L^2$ and~(2) at the third step, we get
\[
\widetilde{R}[L,L^2]\widetilde{R} 
= \widetilde{R}[L,L^2]
\subseteq [\widetilde{R}L,L^2] + [\widetilde{R},L^2]L
\subseteq L^2 + [R,L]L
\subseteq L^2.
\]
as desired.
\end{proof}

\begin{thm}
\label{prp:LLFull}
Let $R$ be a ring such that every proper ideal is contained in a prime ideal, and such that $[R,R]$ is full.
Let $L \subseteq R$ be a Lie ideal.
Then the following are equivalent:
\begin{enumerate}
\item
The subgroup $[L,L]$ is full.
\item
The subgroup $[L,L^2]$ is full.
\item
We have $R=L^2$.
\item
We have $[R,R] \subseteq L$.
\end{enumerate}
Moreover, if this is the case, then $[R,R] = [R,L]$.
\end{thm}
\begin{proof}
We first show that~(1) implies~(4).
Assuming that $[L,L]$ is full, we apply \autoref{prp:BasicLie}(3) at the second step to get
\[
R = \widetilde{R}[L,L]\widetilde{R} 
\subseteq L + L^2.
\]
Using the above at the first step, and \autoref{prp:BasicLie}(2)
at the second step, we obtain
\[
[R,R]
\subseteq [R,L+L^2]
\subseteq [R,L]
\subseteq L.
\]

To show that~(2) implies~(3), assume that $[L,L^2]$ is full.
Applying \autoref{prp:BasicLie}(4) at the second step, we get
\[
R = \widetilde{R}[L,L^2]\widetilde{R} 
\subseteq L^2.
\]

To show that~(3) implies~(4), assume that $R=L^2$.
Applying \autoref{prp:BasicLie}(2) at the third step, we get
\[
[R,R]
= [L^2,L^2]
= [R,L^2]
\subseteq [R,L]
\subseteq L.
\]

Next, we show that~(4) implies~(1) and~(2).
Assume that $[R,R] \subseteq L$.
To verify~(1), set $I := \widetilde{R}[[R,R],[R,R]]\widetilde{R}$.
If $I \neq R$, then by assumption we obtain a prime ideal $P \subseteq R$ containing $I$.
Then the quotient~$R/I$ is a prime ring satisfying $[[R/I,R/I],[R/I,R/I]]=\{0\}$, which by \cite[Theorem~2.3]{GarThi25GenByCommutators} implies that $R/I$ is commutative.
This contradicts the fact that commutators are full in $R$, and hence in $R/I$.
Thus, $I=R$, which means that $[[R,R],[R,R]]$ is full, and consequently so is $[L,L]$.

Similarly, to verify~(2), if the ideal $J$ generated by $[[R,R],[R,R]^2]$ is proper, then it is contained in a prime ideal $Q$ such that $[[Q/J,Q/J],[Q/J,Q/J]^2]=\{0\}$, which also implies that $Q/J$ is commutative by \cite[Theorem~2.3]{GarThi25GenByCommutators}, a contradiction.
This shows that $[[R,R],[R,R]^2]$ is full, and hence so is $[L,L^2]$.

Finally, we have seen in the proof of `(1)$\Rightarrow$(4)' that $[R,R] \subseteq [R,L]$.
\end{proof}

\begin{cor}
Let $R$ be a ring such that every proper ideal is contained in a maximal ideal.  
If $L$ is a Lie ideal of $R$ and if $[L, L]$ is full, then $R=L^2$.
\end{cor}
\begin{proof}
Using that $[L, L]$ is full, it follows that $R$ is idempotent, and thus every maximal ideal of $R$ is a prime ideal by 
\autoref{prp:MaxVsPrime}.
Hence every proper ideal of~$R$ is contained in a prime ideal of $R$. 
It follows from \autoref{prp:LLFull} that $R=L^2$.
\end{proof}

\begin{cor}
\label{prp:GenByCommutators}
Let $R$ be a ring such that every proper ideal is contained in a prime ideal.
Then the following are equivalent:
\begin{enumerate}
\item
The subgroup $[[R,R],[R,R]]$ is full. 
\item
The subgroup $[[R,R],[R,R]^2]$ is full.
\item
We have $R = [R,R]^2$.
\item
The commutator subgroup $[R,R]$ is full. 
\end{enumerate}
Moreover, if this is the case, then $[R,R] = [R,[R,R]]$.
\end{cor}
\begin{proof}
We note that each of the conditions~(1), (2) and~(3) implies that $[R,R]$ is full.
Hence, their equivalence follows from \autoref{prp:LLFull} applied for the Lie ideal $L=[R,R]$.
Further, condition~(3) clearly implies~(4).
Next, assuming that~(4) holds, it also follows from \autoref{prp:LLFull} (again applied for the Lie ideal $L=[R,R]$) that $[[R,R],[R,R]]$ is full, which verifies~(1).
\end{proof}

In the setting of \autoref{prp:LLFull}, we saw conditions characterizing when $[L,L]$ is full for a Lie ideal $L$ in a ring $R$.
Clearly, these conditions also imply that $L$ is fully noncentral, that is, $[R,L]$ is full.
However, the following examples show that the converse does not hold.

\begin{exa}
\label{exa:ExceptionalMatrixAlg}
Let $\mathbb{F}$ be a field of characteristic~$2$, and consider the simple, unital ring $R=M_2(\mathbb{F})$ of $2$-by-$2$ matrices over $\mathbb{F}$.
Consider the subgroup
\[
L= \left\{ 
\begin{pmatrix}
a & b \\
b & a \\
\end{pmatrix}\colon a,b \in \mathbb{F}
\right\}.
\]

One readily sees that $[R,L]=L$, and thus $L \nsubseteq Z(R)$, while $[L,L]=\{0\}$ and  $L^2=L \neq R$.
We also have $[R,R] \nsubseteq L$.
In particular, $L$ is a fully noncentral Lie ideal that does not generate $R$ as a ring, and that does not contain the commutator subgroup.
In light of \autoref{prp:CaseRRSubmodule}, we also compute the (higher) commutator subgroups of $R$ as:
\[
R^{(1)} 
= [R,R]
= \left\{ 
\begin{pmatrix}
a & b \\
c & a \\
\end{pmatrix}\colon a,b,c \in \mathbb{F}
\right\},
\quad
R^{(2)} 
= Z(R), \andSep
R^{(3)} 
= \{0\}.
\]
In particular, $R$ is fully noncentral, while $[R^{(2)},R^{(3)}]=\{0\}$.
\end{exa}

\begin{exa}
\label{exa:ExceptionalDivAlg}
Let $R$ be a $4$-dimensional central division algebra of characteristic~$2$. 
Then there exists a Lie ideal $L$ of $R$ such that neither $[R, R]\subseteq L$ nor $L\subseteq Z(R)$.
To see this, let $K$ be a maximal subfield of $R$, and let $\{1,\mu\}$ be a basis of $K$ over~$Z(R)$. 
Then $R\otimes_{Z(R)}K\cong \text{\rm M}_2(K)$.

Given a Lie ideal $\widetilde{L}$ of the $K$-algebra $R\otimes_{Z(R)}K$, set
\[
L = \big\{ r\in R\colon \mbox{ there is } s\in R \mbox{ such that }r\otimes 1 + s\otimes \mu \in \widetilde{L} \big\}.
\]
Clearly, $L$ is a Lie ideal of the $Z(R)$-algebra $R$. 
Moreover, $\dim_{Z(R)}L=\dim_K\widetilde {L}$.
In particular, if $\dim_K\widetilde {L}=2$ (as in \autoref{exa:ExceptionalMatrixAlg}), then neither $[R, R]\subseteq L$ nor $L\subseteq Z(R)$.
\end{exa}

\section{Rings where proper ideals are contained in non-exceptional prime ideals}
\label{sec:CofinalNonexcPrimeIdls}

To extend the results of \autoref{sec:CofinalPrimeIdeals} to fully noncentral Lie ideals and $[R,R]$-submodules, we need to exclude prime rings as in Examples~\ref{exa:ExceptionalMatrixAlg} and~\ref{exa:ExceptionalDivAlg}.
These are exactly the prime rings where Herstein's techniques for the study of Lie ideals break down (see, for example, \cite[p.120]{LanMon72LieStrPrimeChar2}), and we call them \emph{exceptional};
see \autoref{dfn:Exceptional}.

In this section, we study rings where every proper ideal is contained in a non-exceptional prime ideal.
We show that an $[R,R]$-submodule $V$ in such a ring is fully noncentral if and only if it contains $[R,R]$;
see \autoref{prp:CharFNoncentralRRSubmod}.
It follows that a fully noncentral subgroup $V$ is an $[R,R]$-submodule if and only if it is a Lie ideal, if and only if $[R,R] \subseteq V$;
see \autoref{prp:TFAE-FNSubgp}.

We note that in an algebra over a field of characteristic $\neq 2$, every prime ideal is automatically non-exceptional.
In particular, the class of rings where every proper ideal is contained in a non-exceptional prime ideal includes every \ca{} \cite{GarThi24PrimeIdealsCAlg}.

\medskip

For $n\geq 1$, we write $S_n$ for the permutation group on $n$ elements.
The \emph{standard polynomial of degree $n$} (see, for example \cite[Definition~6.10]{Bre14BookIntroNCAlg}) is the polynomial $s_n\in \mathbb{Z}[x_1,\ldots,x_n]$ given by
\[
s_n(x_1,\ldots,x_n)=\sum_{\sigma\in S_n}\mathrm{sign}(\sigma)x_{\sigma(1)}\cdots x_{\sigma(n)}.\]
We will only need the polynomial $s_4$. 
For a ring $R$, we let $S_4(R)$ denote the additive subgroup of~$R$ generated by the elements $s_4(x_1,\ldots,x_4)$ for $x_1,\ldots,x_4 \in R$.
A ring~$R$ is said to \emph{satisfy the polynomial identity $s_4$} if $s_4(R)=\{0\}$.

Recall that a prime ring $R$ is said to have \emph{characteristic~$2$} if $2x=0$ for every $x \in R$. (Equivalently, if $2x=0$ for some nonzero $x \in R$.)

\begin{dfn}
\label{dfn:Exceptional}
We say that a prime ring $R$ is \emph{exceptional} if it has characteristic~$2$ and it satisfies the polynomial identity~$s_4$.
We say that a prime ideal~$P$ in a ring~$S$ is \emph{exceptional} if the prime ring $S/P$ is exceptional.
\end{dfn}

Let $R$ be an exceptional prime ring.
Then $R$ is commutative if and only if~$R$ is a field of characteristic~$2$.
If $R$ is noncommutative, then $R$ is a prime PI-ring and therefore has nonzero center $Z$, the extended centroid of $R$ 
(see \cite[Definition~7.29]{Bre14BookIntroNCAlg})
agrees with the field of fractions $ZZ^{-1}$, and the ring of central quotients $RZ^{-1}$ is isomorphic to the ring  $M_2(ZZ^{-1})$ of $2$-by-$2$ matrices over the field $ZZ^{-1}$;
see \cite[Section~7.9]{Bre14BookIntroNCAlg}.
It follows that a prime ring $R$ is exceptional if and only if it embeds into $M_2(\mathbb{F})$ for some field $\mathbb{F}$ of characteristic~$2$; see \cite[Corollary~7.59]{Bre14BookIntroNCAlg}.
We also note that a prime ring $R$ is exceptional if and only if $2R+\widetilde{R}S_4(R)\widetilde{R} = \{0\}$.

\begin{prp}
\label{prp:ExceptionalPrimeIdls}
Let $R$ be a ring.
Then the following hold:
\begin{enumerate}
\item
If $R = 2R + \widetilde{R}S_4(R)\widetilde{R}$, then every prime ideal of $R$ is non-exceptional.
\item
If every proper ideal of $R$ is contained in a non-exceptional prime ideal, then $R = 2R + \widetilde{R}S_4(R)\widetilde{R}$.
\end{enumerate}

In particular, every proper ideal of $R$ is contained in a non-exceptional prime ideal if and only if $R = 2R + \widetilde{R}S_4(R)\widetilde{R}$ and every proper ideal of $R$ is contained in a prime ideal.
\end{prp}
\begin{proof}
(1) 
Assume that $R = 2R + \widetilde{R}S_4(R)\widetilde{R}$, and let $P$ be a prime ideal of $R$.
Then $R/P = 2(R/P) + \widetilde{R/P}S_4(R/P)\widetilde{R/P} \neq \{0\}$, which implies that $R/P$ is non-exceptional.

(2)
Assume that every proper ideal of $R$ is contained in a non-exceptional prime ideal.
To reach a contradiction, assume that the ideal $I := 2R + \widetilde{R}S_4(R)\widetilde{R}$ is proper.
By assumption, we obtain a non-exceptional prime ideal $P \subseteq R$ containing~$I$.
Then $2(R/P) + \widetilde{R/P}S_4(R/P)\widetilde{R/P} =\{0\}$, which shows that $R/P$ is exceptional, a contradiction.
\end{proof}

Given an additive subgroup $V$ in a ring, recall that $V^{(n)}$ is defined inductively for $n \geq 0$ as $V^{(0)} := V$ and $V^{(n+1)} := [V^{(n)}, V^{(n)}]$.

\begin{lma}
\label{prp:VmVn}
Let $R$ be a non-exceptional prime ring, and let $V \subseteq R$ be an $[R,R]$-submodule with $[R,V] \neq \{0\}$.
Then $[ V^{(m)}, V^{(n)} ]$ is not central for every $m,n \geq 1$.
\end{lma}
\begin{proof}
Let $Z := Z(R)$ denote the center of $R$. 
Since $R$ is prime, any element $a\in R$ which satisfies $[a,R]\subseteq Z$ will automatically belong to $Z$.
Indeed, the primeness of~$R$ implies that either $Z=\{0\}$ or $Z$ is a domain.
Assuming that $[a,R] \subseteq Z$, and given $x \in R$, we have $a[a,x]=[a,ax] \in Z$ and so $xa[a,x]=a[a,x]x=ax[a,x]$. 
Thus $[a,x][a,x]=0$, implying $[a, x]=0$. So $a\in Z$. 

We claim that $[R,R] \nsubseteq Z$.
Assuming by contradiction that $[R,R] \subseteq Z$, the above observation implies that $R\subseteq Z$, and thus $[R,V]=\{0\}$, which is a contradiction. 
Thus, we have $[R,R] \nsubseteq Z$, as desired.

Since 
\[
V^{(m+n+1)} =\big[ V^{(m+n)}, V^{(m+n)}\big] \subseteq \big[ V^{(m)}, V^{(n)}\big],
\]
it suffices to verify that $V^{(n+1)}$, which equals $\big[V^{(n)}, V^{(n)}\big]$ by definition, is not central for every $n\geq 1$. Using induction
over $n$, we will show that in fact $V^{(n)}\nsubseteq Z$ for every $n\geq 0$. 

The case $n=0$ is true by assumption. Assume that $V^{(n)}\nsubseteq Z$,
and, in order to reach a contradiction, assume that $V^{(n+1)}\subseteq Z$. 
Applying \cite[Lemma~11]{LanMon72LieStrPrimeChar2} for $U = [R,R]$ and $G = V^{(n)}$, we have $[U,G] \subseteq G$ (since $V^{(n)}$ is an $[R,R]$-submodule) and $[G,G]=[V^{(n)},V^{(n)}]=V^{(n+1)}\subseteq Z$.
Since $[R,R] = U \nsubseteq Z$, we deduce that $V^{(n)} = G \subseteq Z$, a contradiction. This finishes the proof.
\end{proof}

In \cite[Theorem~1.2]{Lee22AddSubgpGenNCPoly} it is shown that if $R$ is a ring with $R=2R$ such that every proper ideal is contained in a maximal ideal, and if $L \subseteq R$ is a Lie ideal such that $[R,L]$ is full, then $R = [R,L] + [R,L]^2$ and $[R,R] \subseteq L$.
Next, we show that this result also holds for $[R,R]$-submodules
instead of Lie ideals.
Moreover, we can relax the condition that proper ideals are contained in maximal ideals to containment in prime ideals.
Further, we conclude that $R = [R,L]^2$, that is, the summand $[R,L]$ is not necessary.

\begin{thm}
\label{prp:CharFNoncentralRRSubmod}
Let $R$ be a ring such that every proper ideal is contained in a non-exceptional prime ideal, and such that $[R,R]$ is full.
Let $V \subseteq R$ be an $[R,R]$-submodule.
Then the following are equivalent:
\begin{enumerate}
\item
The subgroup $V$ is fully noncentral.
\item
The subgroup $V^{(n)}$ is full for some (equivalently, every) $n \geq 1$.
\item
The subgroup $[ V^{(m)}, V^{(n)} ]$ is full for some (equivalently all) $m,n \geq 1$
\item
We have $[R,R] \subseteq V$ (and, in particular, $V$ is a Lie ideal).
\end{enumerate}

Moreover, if $V$ is fully noncentral, then $[R,R] = [R,V] = [V,V] = V^{(n)}=R^{(n)}$ for every $n \geq 1$, and $R=V^2$.
\end{thm}
\begin{proof}
We begin by showing that (3) implies (1).
Given $m,n \geq 1$, we can apply \autoref{prp:T-RRSubmodule} as in \autoref{rmk:V2V3-full} to see that
\[
[ V^{(m)}, V^{(n)} ]
\ \subseteq \ [ V^{(1)}, V^{(1)} ]
\ = \ V^{(2)}
\ \subseteq \ V^{(1)}
\ = \ [V,V]
\ \subseteq \ [R,V].
\]
Thus, if $[ V^{(m)}, V^{(n)} ]$ is full for some $m,n \geq 1$, then $V$ is fully noncentral.

Conversely, assume that $[R,V]$ is full, and let $m,n \geq 1$.
Set $I := \widetilde{R}[ V^{(m)}, V^{(n)} ]\widetilde{R}$ and let us verify that $I = R$.
To reach a contradiction, assume that $I \neq R$.
By assumption, we obtain some non-exceptional prime ideal $P \subseteq R$ such that $I \subseteq P$.
Then $R/P$ is a non-exceptional prime ring.
Let $W$ denote the image of $V$ in $R/P$.
Then $W$ is a additive subgroup of~$R/P$ with $[[R/P,R/P],W] \subseteq W$.
Since $[R,V]$ is full, so is $[R/P,W]$.
However, since $I \subseteq P$, we have $[W^{(m)},W^{(n)}]=\{0\}$, which contradicts \autoref{prp:VmVn}.
This shows that $I = R$, as desired.
Thus~(1) and~(3) are equivalent.

Using \autoref{prp:T-RRSubmodule}, one shows that~(2) and~(3) are equivalent.
By \autoref{prp:CaseRRSubmodule}, (3) implies~(4).
Conversely, to see that~(4) implies~(1), assume that $[R,R] \subseteq V$.
Using the last statement in \autoref{prp:GenByCommutators} at the first step, we get
\[
[R,R] 
= [R,[R,R]]
\subseteq [R,V],
\]
and thus $[R,V]$ is full.
Thus, (1)-(4) are equivalent.

\medskip

It remains to show the last statement. We first verify that $[R,R]=R^{(n)}$ for every $n \geq 1$.
By \autoref{prp:GenByCommutators}, we have $[R,[R,R]]=[R,R]$, which shows that $[R,R]$ is fully noncentral.
Given $n \geq 1$, applying the equivalence of~(1) and~(3) for $V=[R,R]$, we get that $[R^{(1)},R^{(n)}]$ is full.
Since $[R^{(1)},R^{(n)}] \subseteq [R,R^{(n)}]$, we deduce that $R^{(n)}$ is fully noncentral, and thus $[R,R] \subseteq R^{(n)}$.
The converse inclusion always holds, which shows that $[R,R]=R^{(n)}$. 

\medskip

Next, assume that $V$ is fully noncentral.
Using~(4), it follows that
\[
[R,R]
= R^{(2)}
= [[R,R],[R,R]]
\subseteq [V,V],
\]
and thus $[R,R]=[R,V]=[V,V]$.
This implies that $[R,R] = R^{(n)} = V^{(n)}$ for all $n \geq 1$, and since $V$ is a Lie ideal we also get $R=L^2$ by \autoref{prp:LLFull}.
\end{proof}

\begin{cor}
\label{prp:TFAE-FNSubgp}
Let $R$ be a ring such that every proper ideal is contained in a non-exceptional prime ideal, and let $V \subseteq R$ be a fully noncentral, additive subgroup.
Then the following are equivalent:
\begin{enumerate}
\item
We have $[R,R] \subseteq V$.
\item
The subgroup $V$ is a Lie ideal.
\item
The subgroup $V$ is an $[R,R]$-submodule.
\end{enumerate}

In particular, if $V$ is a fully noncentral Lie ideal in $[R,R]$, then $V=[R,R]$.
\end{cor}
\begin{proof}
It is clear that~(1) implies~(2), and that~(2) implies~(3).
To see that~(3) implies~(1), we note that $[R,R]$ is full (since $[R,V]$ is full), and thus the fact that~(1) implies~(4) in \autoref{prp:CharFNoncentralRRSubmod} gives the result.
\end{proof}

\begin{rmk}
In the setting of \autoref{prp:CharFNoncentralRRSubmod}, it is not clear that $R=V^2$ implies that~$V$ is fully noncentral.
This should be contrasted with \autoref{prp:LLFull}, where under fewer assumptions on the ring $R$ it was shown that a Lie ideal $L$ is fully noncentral whenever $R=L^2$.
We do not know if a similar result holds in the context of \autoref{prp:CharFNoncentralRRSubmod}, since it is unclear if an $[R,R]$-submodule $V$ satisfying $R=V^2$ is a Lie ideal.
\end{rmk}

\begin{cor}
\label{prp:GenByCommutatorsNonexceptional}
Let $R$ be a ring such that every proper ideal is contained in a non-exceptional prime ideal.
Then the following are equivalent:
\begin{enumerate}
\item
The subgroup $[R,R]$ is full. 
\item
The subgroup $R^{(n)}$ is full for some (equivalently, every) $n \geq 1$.
\item
We have $R = [R,R]^2$.
\end{enumerate}
Moreover, if this is the case, then $[R,R] = [R^{(m)},R^{(n)}]$ for all $m,n \geq 0$, and in particular $[R,R] = R^{(n)}$ for all $n \geq 1$.
\end{cor}
\begin{proof}
It is clear that~(2) implies~(1), since $R^{(n)}\subseteq [R,R]$.
The equivalence of~(1) and~(3) was shown in \autoref{prp:GenByCommutators}.
By applying \autoref{prp:CharFNoncentralRRSubmod} with $V=[R,R]$, it follows that $[R,R]=R^{(n)}$ for all $n \geq 1$.
This shows that~(1)-(3) are equivalent.

Assume that $[R,R]$ is full.
As noted above, we then have $[R,R]=R^{(n)}$ for all $n \geq 1$.
Given $m,n \geq 0$, without loss of generality assume that $m \leq n$.
Then
\[
[R,R] 
= R^{(n+1)}
= [R^{(n)},R^{(n)}]
\subseteq [R^{(m)},R^{(n)}].
\]
The converse inclusion is clear.
\end{proof}

We also obtain the following result for pairs of Lie ideals, which in the setting of \ca{s} was obtained in \cite[Theorem~3.6]{GarThi24PrimeIdealsCAlg}.

\begin{thm}
\label{prp:TwoLieIdls}
Let $R$ be a ring such that every proper ideal is contained in a non-exceptional prime ideal, and let $K,L \subseteq R$ be Lie ideals.
Then the following are equivalent:
\begin{enumerate}
\item
The Lie ideal $[K,L]$ is full.
\item
The Lie ideals $[K,K]$ and $[L,L]$ are full.
\item
The Lie ideals $[R,K]$ and $[R,L]$ are full.
\end{enumerate}
Moreover, if this is the case, then $[R,R] = [K,L] = [K,K]=[L,L]$.
\end{thm}
\begin{proof}
It is clear that~(1) implies~(3) and that~(2) implies~(3).
To show that~(3) implies~(1) and~(2), assume that $K$ and $L$ are fully noncentral.
By \autoref{prp:CharFNoncentralRRSubmod}, we get $[R,R] \subseteq K$ and $[R,R] \subseteq L$.
Using that $R^{(2)}=R^{(3)}$ at the first step (which is true by \autoref{prp:GenByCommutatorsNonexceptional}), we get
\[
[R,R]
= [[R,R],[R,R]]
\subseteq [K,L],
\]
and thus $[R,R]=[K,L]$.
Similarly, we get $[R,R]=[K,K]=[L,L]$.
This shows that $[K,L]$, $[K,K]$ and $[L,L]$ are full, and equal 
to each other.
\end{proof}

The results of this section are applicable in many concrete cases of interest.
We point out the following special case, which seems to not have appeared in the literature:

\begin{cor}
Let $R$ be a unital algebra over a field of characteristic $\neq 2$, and let $V \subseteq R$ be a fully noncentral $[R,R]$-submodule.
Then $R = V^2$ and $[R,R] \subseteq V$.
In particular, $V$ is a Lie ideal.
\end{cor}
\begin{proof}
Since $R$ is unital, every proper ideal is contained in a proper (and maximal) ideal.
Further, since $R$ is an algebra over a field of characteristic $\neq 2$, no prime quotient of $R$ is exceptional.
The result then follows from \autoref{prp:CharFNoncentralRRSubmod}.
\end{proof}

\section{Fully noncentral, invariant subspaces}
\label{sec:FNoncentralSubspaces}

Given a fully noncentral subspace $V$ in an algebra $A$ over a field $\mathbb{F}$, we study the connection between the invariance of $V$ under inner automorphisms of $A$ and the Lie ideal property for $V$. We do this in two different settings:
we assume that every commutator is a sum of either square-zero 
elements (\autoref{prp:TFAE-FNSubspSqZero}), or of nilpotent
elements (\autoref{prp:TFAE-FNSubspNilpot}).

As a first step, we connect invariance of $V$ under inner automorphism induced by square-zero elements to the condition $[x,V] \subseteq V$ for $x^2=0$; 
see \autoref{prp:InvSqZero}. For nilpotent elements instead of
square-zero elements, this is done in \autoref{prp:InvNilpot}.
We then combine these lemmas with 
the results from \autoref{sec:CofinalNonexcPrimeIdls} to obtain
the main results \autoref{prp:TFAE-FNSubspSqZero}
and \autoref{prp:TFAE-FNSubspNilpot}.

The condition that every commutator is a sum of square-zero elements is automatically satisfied in a number of cases of 
interest, for example for the class of zero-product balanced algebras introduced in \cite{GarThi23ZeroProdBalanced};
see \autoref{prp:Balanced}.

\medskip

Given an algebra $A$, we say that an automorphism $\alpha \colon A \to A$ is inner if there exists an invertible element $v \in \widetilde{A}$ such that $\alpha(a) = vav^{-1}$ for all $a \in A$.
We actually only consider automorphisms of the form $a \mapsto (1+x+x^2+\ldots+x^{k-1})a(1-x)$ induced by a nilpotent element $x \in A$ with $x^k = 0$.

We set $N_2(A) := \{ x \in A : x^2=0 \}$, the set of square-zero elements in $A$.
We let $N_2(A)^+$ denote the subspace of $A$ generated by the square-zero elements. We will say that a subset $V\subseteq A$ is
$N_2(A)$-\emph{invariant}, if it is invariant under all inner
automorphisms induced by elements from $N_2(A)$. Equivalently, 
if $a\in V$ and $x^2=0$ imply $(1+x)a(1-x)\in V$.

The following result is probably known, but we could not locate it in the literature.

\begin{lma}
\label{prp:InvSqZero}
Let $A$ be an algebra over a field $\mathbb{F}$, and let $V \subseteq A$ be a subspace.
Then the following statements hold:
\begin{enumerate}
\item
If $\mathbb{F} \neq \{0,1\}$ and $V$ is $N_2(A)$-invariant, then $[N_2(A)^+,V] \subseteq V$.
\item
If $\charact(\mathbb{F}) \neq 2$ and $[N_2(V)^+,V] \subseteq V$, then $V$ is $N_2(A)$-invariant.
\end{enumerate}
\end{lma}
\begin{proof}
(1)
Let $a \in V$ and $x \in N_2(A)$.
We need to verify $[x,a] \in V$.

We have $(1+x)V(1-x) \subseteq V$, and thus
\[
[x,a] - xax
= (1+x)a(1-x) - a \in V.
\]

Since $\mathbb{F}$ contains at least three elements, we can choose $\lambda \in F$ with $\lambda \neq 0,1$.
Then $\lambda x$ is also a square-zero element, and applying the above to $\lambda x$ instead of $x$ we obtain
\[
\lambda [x,a] - \lambda^2 xax
= (1+\lambda x)a(1-\lambda x) - a \in V.
\]

Since $V$ is a subspace, we obtain that $\lambda^2([x,a] - xax) \in V$, and thus
\[
(\lambda^2-\lambda)[x,a]
= \lambda^2([x,a] - xax) - \lambda [x,a] + \lambda^2 xax 
\]
also belongs to $V$.
Since $(\lambda^2-\lambda)\neq 0$, and using again that $V$ is a subspace, we get $[x,a] \in V$.

(2)
Let $a \in V$ and $x \in N_2(A)$.
We need to verify $(1+x)a(1-x) \in V$.
We have
\[
(1+x)a(1-x)
= a + [x,a] - xax.
\]
Since $[x,a] \in V$, it suffices to verify that $xax \in V$.
To see this, note that
\[
-2xax = [x,[x,a]]
\]
belongs to $V$, since $[x,a]\in V$.
Since $\charact(\mathbb{F}) \neq 2$ and since $V$ is a subspace, we get $xax \in V$.
\end{proof}

\begin{thm}
\label{prp:TFAE-FNSubspSqZero}
Let $A$ be an algebra over a field $\neq \{0,1\}$ such that every commutator in $A$ is a sum of square-zero elements, and every proper ideal is contained in a non-exceptional prime ideal.
Let $V \subseteq A$ be a fully noncentral subspace.
Then the following are equivalent:
\begin{enumerate}
\item
The subspace $V$ is invariant under all inner automorphisms of~$A$.
\item
The subspace $V$ is invariant under all inner automorphisms induced by square-zero elements of~$A$.
\item
We have $[A,A] \subseteq V$.
\item
The subspace $V$ is a Lie ideal.
\item
The subspace $V$ is an $[A,A]$-submodule.
\end{enumerate}
\end{thm}
\begin{proof}
It is clear that~(1) implies~(2).
By \autoref{prp:TFAE-FNSubgp}, (3)-(5) are equivalent.
Let us show that~(2) implies~(5).
By \autoref{prp:InvSqZero} we have $[N_2(A)^+,V] \subseteq V$, 
and since $[A,A] \subseteq N_2(A)^+$ by assumption, we get $[[A,A],V] \subseteq V$.

Finally, to show that~(3) implies~(1), let $u$ be an invertible element in $\widetilde{A}$, and let $x \in V$.
We have $[\widetilde{A},\widetilde{A}]=[A,A] \subseteq V$, and therefore
\[
uxu^{-1}
= [u,xu^{-1}] + x
\in [\widetilde{A},\widetilde{A}] + V 
\subseteq V.
\]
This shows that $uVu^{-1} \subseteq V$, as desired.
\end{proof}

For an algebra $A$, we write $N(A)$ for the set of its nilpotent
elements, and $N(A)^+$ for the subspace it generates.
The next result is a variation of \autoref{prp:InvSqZero},
using nilpotent elements instead of square-zero elements.

\begin{lma}
\label{prp:InvNilpot}
Let $A$ be an algebra over a field $\mathbb{F}$, and let $V \subseteq A$ be a subspace.
Then the following statements hold:
\begin{enumerate}
\item
If $\mathbb{F}$ is infinite and $V$ is $N(A)$-invariant, then $[N(A)^+,V] \subseteq V$.
\item
If $\charact(\mathbb{F}) = 0$ and $[N(A)^+,V] \subseteq V$, then $V$ is $N(A)$-invariant.
\end{enumerate}
\end{lma}
\begin{proof}
(1)
Let $a \in V$ and $k \geq 2$.
We need to verify that the commutator $[x,a]$ belongs to $V$ whenever $x\in A$ satisfies $x^k=0$.
Given such an element $x$, we have $(1+x+x^2+\ldots+x^{k-1})V(1-x) \subseteq V$, and thus
\[
[x,a] + x[x,a] + \ldots + x^{k-1}[x,a]
= (1+x+x^2+\ldots+x^{k-1})a(1-x) - a \in V.
\]
Given a nonzero $\lambda \in \mathbb{F}$, using that $(\lambda x)^k=0$, we get
\begin{align*}
&[x,a] + \lambda x[x,a] + \ldots + \lambda^{k-1} x^{k-1}[x,a] \\
&\qquad= \lambda^{-1}\big( [\lambda x,a] + \lambda x[\lambda x,a] + \ldots + (\lambda x)^{k-1}[\lambda x,a] \big)
\in V.
\end{align*}

Using that $\mathbb{F}$ is infinite, we can choose elements $\lambda_1,\ldots,\lambda_{k-1} \in F$ that are pairwise distinct and nonzero.
Then the corresponding Vandermonde matrix
\[
\begin{pmatrix}
1 & \lambda_1 & \lambda_1^2 & \cdots & \lambda_1^{k-1} \\
1 & \lambda_2 & \lambda_2^2 & \cdots & \lambda_2^{k-1} \\
\vdots & & & & \vdots \\
1 & \lambda_{k-1} & \lambda_{k-1}^2 & \cdots & \lambda_{k-1}^{k-1} \\
\end{pmatrix}
\]
is invertible and using the entries in the first column of the inverse matrix, we obtain $\mu_1, \ldots, \mu_{k-1} \in \mathbb{F}$ such that
\[
[x,a]
= \sum_{j=1}^{k-1} \mu_j \big( [x,a] + \lambda_j x[x,a] + \ldots + \lambda_j^{k-1} x^{k-1}[x,a] \big) \in V.
\]
This finishes the proof.

(2)
Let $x\in A$ satisfy $x^{k+1}=0$ and let $a\in V$. 
Set $b=(1-x)a(1+x+\cdots+x^k)$. 
Then
\begin{align*}
  b&=(1-x)a(1+x+\cdots+x^k)\\
   &=a-[x, a]-xa(x+\cdots+x^k)+a(x^2+\cdots+x^k) \\
  &= a-[x, a]-\sum_{i=1}^{k}[x, a]x^j.
\end{align*}
We aim to prove that $b\in V$. For this, it suffices to show that $[x, a]x^j\in V$ for all $j=1,\ldots,k$.
Since
$[x, [x, a]]$ belongs to $V$ as $[N(A),V]\subseteq V$, we deduce that also $x^2a-2xax+ax^2$ belongs to  $V$. Also, $x^2a-ax^2\in V$.
We get
\begin{align}\label{eq:3}\tag{5.1}
ax^2-xax=[a, x]x\in V.
\end{align}
Note that
\[
[x, [x, [x, a]]]=x^3a-3x^2ax+3xax^2-ax^3=[x^3, a]-3x[x, a]x
\]
belongs to $V$ as $[x, [x, a]]\in V$ and $[N(A),V]\subseteq V$, and thus $x[x, a]x\in V$. Note that 
\[
[x, [a, x]x]\stackrel{\eqref{eq:3}}{=}x[a, x]x-[a, x]x^2
\]
belongs to $V$, and thus 
\begin{align}
\label{eq:4}\tag{5.2}
[a, x]x^2\in V.
\end{align}

Let $x\in N(A)$.
Given $n\geq 2$, we assume that $[a, x]x^j\in V$ for all $j=1,\ldots,n$. We claim that $[a, x]x^{n+1}\in V$.
Since $x+x^n$ is also nilpotent, replacing $x$ by $ x+x^n$ in  \eqref{eq:4} gives $[a, x+x^n](x+x^n)^2\in V$. Using this, 
we get
$$
[a, x+x^n](x+x^n)^2=[a, x+x^n](x^2+2x^{n+1}+x^{2n})\in V,
$$
which implies that
$$
[a, x](2x^{n+1}+x^{2n})+[a, x^n](x^2+2x^{n+1})\in V.
$$
Hence
\begin{align}\label{eq:5}\tag{5.3}
2[a, x]x^{n+1}+[a, x^n]x^2\in V.
\end{align}
We compute
\begin{align*}
[a, x^n]x^2&=\sum_{j=0}^{n-1}x^j[a, x]x^{n-j-1}x^2\\
&=\sum_{j=1}^{n-1}x^j[a, x]x^{n-j+1}+[a, x]x^{n+1}\\
&=\sum_{j=1}^{n-1}\Big[x^j, [a, x]x^{n-j+1}\Big]+(n-1)[a, x]x^{n+1}+[a, x]x^{n+1}\\
&=\sum_{j=1}^{n-1}\Big[x^j, [a, x]x^{n-j+1}\Big]+n[a, x]x^{n+1}
\end{align*}
Using \eqref{eq:5} at the last step, we get
\[
2[a, x]x^{n+1}+[a, x^n]x^2=\sum_{j=1}^{n-1}\Big[x^j, [a, x]x^{n-j+1}\Big]+(n+2)[a, x]x^{n+1}\in V.\]
Since ${n-j+1}\leq n$ for $j=1,\ldots, n-1$, we get $\Big[x^j, [a, x]x^{n-j+1}\Big]\in V$ and so
$$
[a, x]x^{n+1}\in V,
$$
as desired.
\end{proof}

\begin{prp}
\label{prp:CommutatorNilpotents}
Let $A$ be an algebra over an infinite field.
Then, given nilpotent elements $x,y \in A$, the commutator $[x,y]$ is a finite sum of nilpotent elements.
\end{prp}
\begin{proof}
This follows from part~(1) of~\autoref{prp:InvNilpot} applied to $V=N(A)^+$, since the set consisting of finite sums of nilpotent elements is a subspace that is invariant under all (inner) automorphisms.
\end{proof}

The next result is a variation of \autoref{prp:TFAE-FNSubspSqZero}.
By strengthening the assumption on the field, we can relax the assumption that every commutator is a sum of square-zero elements to allow nilpotent elements instead.

\begin{thm}
\label{prp:TFAE-FNSubspNilpot}
Let $A$ be an algebra over an infinite field such that every commutator in $A$ is a sum of nilpotent elements, and such that every proper ideal is contained in a non-exceptional prime ideal.
Let $V \subseteq A$ be a fully noncentral subspace.
Then the following are equivalent:
\begin{enumerate}
\item
The subspace $V$ is invariant under all inner automorphisms of~$A$.
\item
The subspace $V$ is invariant under all inner automorphisms induced by nilpotent elements of~$A$.
\item
We have $[A,A] \subseteq V$.
\item
The subspace $V$ is a Lie ideal.
\item
The subspace $V$ is an $[A,A]$-submodule.
\end{enumerate}
\end{thm}
\begin{proof}
The result is proved analogous to \autoref{prp:TFAE-FNSubspSqZero}, with the only difference that for the implication `(2)$\Rightarrow$(5)' we use that $[A,A] \subseteq N(A)^+$ by assumption, that $[N(A)^+,V] \subseteq V$ by \autoref{prp:InvNilpot}, and consequently $[[A,A],V] \subseteq V$.
\end{proof}

The next result can also be deduced from \cite[Theorem~2]{Chu99InvSubgpSimpleRg}.

\begin{exa}
Let $R$ be a noncommutative, simple ring with infinite center and such that every commutator is a sum of nilpotent elements.
Assume that $R$ is not an exceptional prime ring (that is, $R$ does not embed into $M_2(\mathbb{F})$ for a field $\mathbb{F}$ of characteristic $2$).
Let $V \subseteq R$ be an additive subgroup that is invariant under all inner automorphisms of~$R$.
Then either $V \subseteq Z(R)$ or $[R,R] \subseteq V$.
In either case, $V$ is a Lie ideal.
Indeed, if $V \nsubseteq Z(R)$, then $V$ is fully noncentral because $R$ is simple, and the result follows from \autoref{prp:TFAE-FNSubspNilpot}.
\end{exa}

An algebra $A$ over a field $\mathbb{F}$ is said to be \emph{zero-product balanced} if for all $x,y,z \in A$, the element $xy\otimes z - x\otimes yz \in A\otimes_{\mathbb{F}} A$ belongs to the subspace of $A \otimes_{\mathbb{F}} A$ generated by $\{v\otimes w : vw=0\}$;
see \cite[Definition~2.6]{GarThi23ZeroProdBalanced}.
This is closely related to the concept of a \emph{zero-product determined} algebra \cite{Bre21BookZeroProdDetermined}, \cite[Section~2]{GarThi23ZeroProdBalanced}.
In particular, a unital algebra is zero-product balanced if and only if it is zero-product determined.

\begin{cor}
\label{prp:Balanced}
Let $A$ be a zero-product balanced, unital algebra over a field of characteristic $\neq 2$.
For a fully noncentral subspace $V \subseteq A$, the following
are equivalent:
\begin{enumerate}
\item 
The subspace $V$ is a Lie ideal.
\item 
The subspace $V$ is invariant under all inner automorphisms of~$A$.
\item 
We have $[A,A] \subseteq V$.
\end{enumerate}
\end{cor}
\begin{proof}
Since $A$ is unital and zero-product balanced (hence zero-product determined), every commutator in $A$ is a sum of square-zero elements by \cite[Theorem~9.1]{Bre21BookZeroProdDetermined};
see also \cite[Theorem~5.3]{GarThi23ZeroProdBalanced}.
Further, since $A$ is unital, every proper ideal is contained in a maximal ideal, and maximal ideals in unital algebras are prime;
see \autoref{prp:MaxVsPrime}.
Finally, since $A$ is an algebra over a field of characteristic $\neq 2$, every prime ideal in $A$ is non-exceptional.
Hence, the result follows from \autoref{prp:TFAE-FNSubspSqZero}.
\end{proof}

\begin{exa}
Let $A$ be a unital algebra over a field of characteristic $\neq 2$, let $n \geq 2$, and let $V \subseteq M_n(A)$ be a fully noncentral subspace.
Then $V$ is a Lie ideal if and only if it is invariant under all inner automorphisms of~$M_n(A)$, and if and only if $[M_n(A),M_n(A)] \subseteq V$.
Indeed, by \cite[Theorem~3.8]{GarThi23ZeroProdBalanced}, $M_n(A)$ is zero-product balanced (see also \cite[Corollary~2.4]{Bre21BookZeroProdDetermined}), whence the result follows from \autoref{prp:Balanced}.
\end{exa}

\section{Commutators as sums of square-zero elements}
\label{sec:SumSqZero}

Many interesting rings have the property that every commutator is a sum of square-zero elements.
In this section, we show that if this is the case, then under mild additional assumptions, the commutator subgroup admits a precise description as the additive subgroup generated by a special class of square-zero elements;
see \autoref{prp:Commutator-FN2}.
It remains open if the same holds under the weaker assumption that every commutator is a sum of nilpotent elements;
see \autoref{qst:Nilpotent}.

\medskip

Given a ring $R$, recall that we set $N_2(R) := \{ x \in R : x^2=0 \}$.
We say that $x \in R$ is an \emph{orthogonally factorizable square-zero element} if there exist $y,z \in R$ such that $x=yz$ and $zy=0$.
We denote the set of such elements by $FN_2(R)$;
see \cite[Definition~5.1]{GarThi23ZeroProdBalanced}.
We have $FN_2(R) \subseteq [R,R]$, since if $x=yz$ and $zy=0$, then $x=[y,z]$.

Given a subset $B$ of $R$, we use $B^+$ to denote the additive subgroup of $R$ generated by $B$.
Recall that a subset in a ring is \emph{full} if it is not contained in any proper ideal.

\begin{lma}
\label{prp:N2-full}
Let $R$ be a ring such that every proper ideal is contained in a prime ideal.
Then the following are equivalent:
\begin{enumerate}
\item
The set $N_2(R)$ is full.
\item
The set $FN_2(R)$ is full.
\item
The set $FN_2(R)$ is fully noncentral.
\end{enumerate}
\end{lma}
\begin{proof}
Since $[R,FN_2(R))]$ is contained in the ideal generated by $FN_2(R)$, we see that~(3) implies~(2).
It is clear that~(2) implies~(1). 

To show that~(1) implies~(3), let $I$ denote the ideal of $R$ generated by $[R,FN_2(R)]$.
Assuming that $I \neq R$, choose a prime ideal $P \subseteq R$ such that $I \subseteq P$.
We will show that every square-zero element is contained in $P$.

Let $x \in N_2(R)$.
Note that $xRx \subseteq FN_2(R)$, and thus $[R,xRx] \subseteq P$.
Thus, for each $a \in R$, the image of $xax$ in $R/P$ is a central square-zero element.
Since $R/P$ is a prime ring, its center is either zero or a domain.
Consequently, every central square-zero element in $R/P$ is zero, and we deduce that $xRx \subseteq P$, which implies $x \in P$, as desired.
\end{proof}

\begin{thm}
\label{prp:Commutator-FN2}
Let $R$ be a ring such that every proper ideal is contained in a non-exceptional prime ideal.
Assume that $[R,R] \subseteq N_2(R)^+$, and that $N_2(R)$ is full.
Then
\[
[R,R] = FN_2(R)^+.
\]
\end{thm}
\begin{proof}
Set $V := FN_2(R)^+$.
The inclusion $V \subseteq [R,R]$ holds in general.
We will apply \autoref{prp:TFAE-FNSubgp} to obtain the converse inclusion $[R,R] \subseteq V$.

Since $N_2(R)$ is full, it follows from \autoref{prp:N2-full} that $V$ is fully noncentral.
Next, we show that $[N_2(R)^+,V] \subseteq V$.
Let $x \in N_2(R)$ and $w \in V$.
Since $V$ is invariant under all automorphisms of $R$, we have $(1+x)w(1-x) \in V$.
Further, we have $w \in V$ and $xwx \in FN_2(R) \subseteq V$, and therefore
\[
[x,w] 
= (1+x)w(1-x) - w - xwx 
\in V.
\]

By assumption, we have $[R,R] \subseteq N_2(R)^+$, and it follows that $V$ is an $[R,R]$-submodule.
Applying \autoref{prp:TFAE-FNSubgp} we get $[R,R] \subseteq V$.
\end{proof}

If $R = M_2(\mathbb{F})$ for a field $\mathbb{F}$ of characteristic $2$, then $[R,R]$ agrees with the additive subgroup generated by $FN_2(R)$.
This suggests that the answer to the following question might be positive.

\begin{qst}
Does \autoref{prp:Commutator-FN2} also hold for rings where every proper ideal is contained in a (possibly exceptional) prime ideal?
\end{qst}

\begin{cor}
\label{prp:N2-FN2}
Let $R$ be a ring such that every proper ideal is contained in a non-exceptional prime ideal, and assume that $[R,R]$ is full.
Then the following are equivalent:
\begin{enumerate}
\item
We have $[R,R] \subseteq N_2(R)^+$.
\item
We have $[R,R] = FN_2(R)^+$.
\end{enumerate}
\end{cor}
\begin{proof}
It is clear that~(2) implies~(1).
Conversely, assume that every commutator is a sum of square-zero elements.
Since $[R,R]$ is full, it follows that the set of square-zero elements is full, and now~(2) follows from \autoref{prp:Commutator-FN2}.
\end{proof}

\begin{rmks}
(1)
If $R$ is a zero-product balanced, idempotent ring, then $[R,R] = FN_2(R)^+$ by \cite[Theorem~5.3]{GarThi23ZeroProdBalanced}.
This includes all rings generated by idempotents \cite[Example~3.7]{GarThi23ZeroProdBalanced}, in particular all simple rings that contain a nontrivial idempotent.
It also includes all unital \ca{s} that have no one-dimensional irreducible representations \cite{GarThi25pre:ZeroProdRingsCAlgs}.
The famous Jiang-Su algebra is a unital, simple \ca{} that contains no idempotents other than zero and one.
This algebra has no one-dimensional irreducible representations and is therefore zero-product balanced.

(2)
There exist simple rings where $FN_2(R)^+$ is a proper subset of $[R,R]$.
Indeed, by \cite[Theorem~10]{CheLeePuc10CommNilSimpleRings} there exists a simple ring~$R$ such that the nilpotent elements in~$R$ form a subring $W$ with $\{0\} \neq W \neq R$.
Then $R = [R,R]^2$ (see \autoref{prp:GenByCommutators}), but $R$ is not generated by $FN_2(R)$ as a ring.
\end{rmks}

\begin{pbm}
Find rings $R$ such that $FN_2(R)^+$ is a proper subset of $N_2(R)^+$.
\end{pbm}

We end the paper with a short discussion of the relationship between nilpotent elements, square-zero elements, and commutators.
In \ca{s} it is known that every nilpotent element is a sum of commutators (\cite[Lemma~2.1]{Rob16LieIdeals}), and for many \ca{s} it is known that every commutator is a sum of square-zero elements (\cite[Theorem~4.2]{Rob16LieIdeals}, \cite{GarThi25pre:ZeroProdRingsCAlgs}), although it remains open if this holds for every \ca{} (\cite[Question~2.5]{Rob16LieIdeals}, \cite[Question~4.1]{GarThi24PrimeIdealsCAlg}).
This raises the question if every nilpotent element in a \ca{} is a sum of square-zero elements (\cite[Question~4.5]{GarThi24PrimeIdealsCAlg}).

Of course, in general rings it is not true that nilpotent elements are sums of square-zero elements.
Nevertheless, it is conceivable that the answer to the following question may be positive.

\begin{qst}
\label{qst:Nilpotent}
Let $R$ be a ring such that every proper ideal of $R$ is contained in a non-exceptional prime ideal, and such that $[R,R]$ is full.
Assume that every commutator in $R$ is a sum of nilpotent elements.
Does it follow that every commutator in $R$ is a sum of square-zero elements?
\end{qst}


\providecommand{\bysame}{\leavevmode\hbox to3em{\hrulefill}\thinspace}
\providecommand{\noopsort}[1]{}
\providecommand{\mr}[1]{\href{http://www.ams.org/mathscinet-getitem?mr=#1}{MR~#1}}
\providecommand{\zbl}[1]{\href{http://www.zentralblatt-math.org/zmath/en/search/?q=an:#1}{Zbl~#1}}
\providecommand{\jfm}[1]{\href{http://www.emis.de/cgi-bin/JFM-item?#1}{JFM~#1}}
\providecommand{\arxiv}[1]{\href{http://www.arxiv.org/abs/#1}{arXiv~#1}}
\providecommand{\doi}[1]{\url{http://dx.doi.org/#1}}
\providecommand{\MR}{\relax\ifhmode\unskip\space\fi MR }
\providecommand{\MRhref}[2]{%
  \href{http://www.ams.org/mathscinet-getitem?mr=#1}{#2}
}
\providecommand{\href}[2]{#2}

\end{document}